\documentclass[12pt]{amsart}

\usepackage{cancel}
\usepackage{latexsym, amssymb, amsmath,esint}
\usepackage{soul}
\usepackage{amsfonts, graphicx}
\usepackage{graphicx,color}
\usepackage{wasysym,stmaryrd}
\newcommand{\be}{\begin{equation}}
\newcommand{\ee}{\end{equation}}
\newcommand{\beq}{\begin{eqnarray}}
\newcommand{\eeq}{\end{eqnarray}}

\newtheorem{thm}{Theorem}[section]

\newtheorem{lma}{Lemma}[section]

\theoremstyle{remark}

\numberwithin{equation}{section}

\def\be{\begin{equation}}
\def\ee{\end{equation}}
\def\bee{\begin{equation*}}
\def\eee{\end{equation*}}
\def\ol{\overline}
\def\lf{\left}
\def\ri{\right}

\def\vn{{\vec\nu}}

\def\bn{\mathbf{n}}

\def\Ric{\text{\rm Ric}}

\def\cK{\mathcal{K}}

\def\wt{\widetilde}
\def\la{\langle}
\def\ra{\rangle}
\def\p{\partial}

\def\ol{\overline}

\def\e{\varepsilon}
\def\a{{\alpha}}
\def\b{{\beta}}

\def\R{\mathbb{R}}

\def\oRm{{\overline{\mathrm{Rm}}}}

\def\n{\nabla}

\def\on{{\overline{\nabla}}}
\def\oR{{\overline{R}}}

\begin{document}

\title[]
{ Some  estimates on stable minimal hypersurfaces in Euclidean space}

\author{Luen-Fai Tam}
\address[Luen-Fai Tam]{The Institute of Mathematical Sciences and Department of Mathematics, The Chinese University of Hong Kong, Shatin, Hong Kong, China.}
 \email{lftam@math.cuhk.edu.hk}

\renewcommand{\subjclassname}{
  \textup{2020} Mathematics Subject Classification}
\date{\today}

\begin{abstract} We derive some estimates for stable minimal hypersurfaces in $\R^{n+1}$. The estimates are related to recent proofs of Bernstein theorems for complete stable minimal hypersurfaces in $\R^{n+1}$ for $3\le n\le 5$ by Chodosh-Li \cite{ChodoshLi2023}, Chodosh-Li-Minter-Stryker \cite{ChodoshLiMinterStryker2024} and Mazet \cite{Mazet2024}. In particular, the estimates indicate that the methods in their proofs may not work for $n=6$,  which is  observed also by Antonelli-Xu \cite{AntonelliXu2024-2} and Mazet \cite{Mazet}. The method of derivation in this work might  also be applied to other problems.

\end{abstract}

\maketitle

\markboth{Luen-Fai Tam}{Some estimates on stable minimal hypersurfaces}
\section{Introduction}\label{s-intro}
In \cite{ChodoshLi2023},   Chodosh-Li proved that a complete stable minimal hypersurface in $\R^{n+1}$ is a hyperplane for $n=3$. In \cite{ChodoshLiMinterStryker2024}, Chodosh-Li-Minter-Stryker proved that this is   true for $n=4$. In \cite{Mazet2024}, Mazet proved that this is true for $n=5$. For   history and development of  Bernstein type results for stable minimal hypersurfaces in Euclidean space, please consult \cite{ChodoshLi2023,ChodoshLiMinterStryker2024,Mazet2024}. Similar to the recent work by Antonelli-Xu \cite{AntonelliXu2024-2}, we want to have a better understanding of the methods in  \cite{ChodoshLi2023,ChodoshLiMinterStryker2024,Mazet2024}, and possible obstructions for these methods in case $n=6$. In this work, we will obtain some estimates similar to those in  \cite{ChodoshLiMinterStryker2024,Mazet2024}  in the form which might be more easy to use.

 To state our main result, let $M^n$ be a complete two-sided stable minimal hypersurface in $\R^{n+1}$. We only consider the case that $3\le n\le 6$. Let $G$ be the standard metric on $\R^{n+1}$. Following \cite{ChodoshLi2023,ChodoshLiMinterStryker2024,Mazet2024}, consider the conformal metric $\ol G=r^{-2}G=e^{-2s}G$, with $s=\log r$, where $r$ is the Euclidean distance from the origin. Let $g, \ol g$ be the metrics on $M$ induced by $G, \ol G$ respectively. Let $w>0$ be the smooth function given by \eqref{e-w} in section \ref{s-mu}. Consider a $\mu$-bubble in $(M, \ol g)$ given by $w^a$ for some $a\ge 1$ and some $h$. More precisely,  let $\Omega_-\subset \Omega_+\subset M$ be  domains with nonempty boundary $\p\Omega_-, \p\Omega_+$   respectively, so that $ \Omega_+\setminus \ol\Omega_-$ is bounded. Let $h$ be a smooth function on $ \Omega_+\setminus \ol\Omega_-$ with $h\to  \infty$ on $\p\Omega_-$ and $h\to-\infty$ on  $\p\Omega_+$. Fix a domain $\Omega_0$ with $\Omega_-\Subset \Omega_0\Subset \Omega_+$ and let $\Sigma$ be a $\mu$-bubble which is the boundary of a minimizer of
\be\label{e-mu-bubble}
E(\Omega)=\int_{\p^*\Omega}w^a-\int (\chi_{\Omega}-\chi_{\Omega_0}) h w^a.
\ee
where $1\le a<4$. We have the following estimates:
\begin{thm}\label{t-main} For $1\le a<4$, and $0\le \a\le1$,
\bee
\begin{split}
\frac4{4-a}\int_\Sigma |\n^\Sigma \phi|^2 d\mu_\Sigma\ge& \int_\Sigma\bigg[(n-2)\lf(\a-\frac{(n-2)a}4\ri)-\a \lambda_\Sigma-|\on h|\\
&\ \ \ \ +P(x,y,\omega)+Q(H,z,h)\bigg] \phi^2 d\mu_\Sigma
\end{split}
\eee
where $\n^\Sigma$ is the derivative of $\Sigma$, $\on$ is the derivative of $(M,\ol g)$, and at each point $\lambda_\Sigma$ is the infimum of $\Ric^\Sigma$ at that point. Here $P,Q$ are quadratic forms given by
\bee
\begin{split}
P(x,y,\omega)=&\lf(n-1 +\lf(\frac {n^2}{n-2}-\frac{ (n^2+4n-4)}4\ri)a\ri)\omega^2
+( \frac {  n-1   }{n-2}a-1)x^2\\
&+(\frac { n-1 }{n-2}a-\a)y^2 +(\frac {2a}{n-2}-\a)xy
+n\lf((\frac {2a}{n-2}-1)x  +(\frac{2a}{n-2}-\a)y \ri)\omega
\end{split}
\eee
for some functions $x, y, \omega$ on $\Sigma$ with $\omega^2\le 1$; and
\bee
\begin{split}
Q(H,z,h)=&\lf( \frac1a-\frac{n-3}{n-2}\ri)H^2+( \frac{n-1}{n-2}-\a)z^2
 +\frac1a h^2\\
 &+\lf(\a-\frac 2{n-2}\ri) zH+\lf(1-\frac2a\ri)hH
\end{split}
\eee
for some functions $H, z$ on $\Sigma$.
\end{thm}

Since $P, Q$ are quadratic forms in three variables, it is relatively easy to estimate $P, Q$ from below. Suppose we can find $1\le a<4, 0\le \a\le 1$ and $h$ such that:
\begin{enumerate}
  \item [(i)] $(n-2)(\a-\frac{(n-2)a}4)>\delta$;
  \item [(ii)] $P\ge 0$;
  \item [(iii)] $Q\ge ch^2$ for some $c>0$;
  \item[(iv)] $ch^2-|\on h|+\e\ge0$ for some $0<\e<\delta$.
\end{enumerate}
Then we have
\bee
\frac4{4-a}\int_\Sigma |\n \phi|^2 d\mu_\Sigma\ge  \int_\Sigma(\delta-\e-\a \lambda_\Sigma)\phi^2 d\mu_\Sigma
\eee
for all smooth function $\phi$ on $\Sigma$. If in addition $\frac{4}{4-a}\le \frac{\a(n-3)}{n-2}$, then one can apply the comparison theorem of Antonelli-Xu \cite{AntonelliXu2024-1} to obtain upper bound of the area of $\Sigma$ and obtain some Bernstein type theorem for stable minimal hypersurfaces. For $3\le n\le 5$, these are results by Chodosh-Li, Chodosh-Li-Minter-Stryker and Mazet \cite{ChodoshLi2023,ChodoshLiMinterStryker2024,Mazet2024}. In case $n=6$, one can see (i) in above cannot be true because $a\ge1\ge\a$. (iii) is also not true   in case that $n=6$. These can be considered   as obstructions to apply   the methods in case $n=6$. Hence the method may not work for $n=6$. This is also observed by by Antonelli-Xu \cite{AntonelliXu2024-2} and Mazet \cite{Mazet}.
 
 On the other hand, $\a$ in the theorem is the same $\a$ in the definition of $\a$-$\mathrm{biRicci}$ introduced in \cite{Mazet2024}. This is actually  a convex combination of bi-Ricci curvature and Ricci curvature which arise naturally in the second variational of $\mu$-bubble, see Lemma \ref{l-bubble-3} for more details. For the definition of biRicci curvature see \eqref{e-biRicci} . When $\a=1$, then $\a$-$\mathrm{biRicci}$ curvature is the same as biRicci curvature of $M$. In case $n=5$ and $\a=1$, then
  $Q\le 0$ if $z=h=H/2$ and so (iii) may  not be true. Hence in this case we need to choose $0<\a<1$ as in Mazet's work \cite{Mazet2024} and we cannot just use bi-Ricci curvature.

  The organization of the paper is as follows. In section \ref{s-basic}, we derive a version of stability condition. In section \ref{s-mu}, we state a property for $\mu$-bubble. In section \ref{s-main}, we will prove the main theorem and in section \ref{s-application} we will indicate some applications of the estimates in the main theorem.

\section{Basic facts}\label{s-basic}

Recall the following, see \cite[Theorem 7.30]{Lee}. Let $(N,g)$ be a Riemannian manifold. Consider the conformal metric $\wt g=e^{2f}g$. Let  $\wt {\mathrm{Rm}},  {\mathrm{Rm}}$  be the curvature tensors of $\wt g, g$ respectively. Then
\be\label{e-curvature-coformal}
\wt {\mathrm{Rm}}=e^{2f}\lf[\mathrm{Rm}-\nabla^2_g f\owedge g+(df\otimes df)\owedge g-\frac12|df|^2_g g\owedge g\ri].
\ee
Here  $\owedge$ be the Kulkarni-Nomizu product. Namely, for symmetric $(0,2)$ tensors $S, T$ on a manifold,
\be\label{e-KNproduct}
\begin{split}
(S\owedge T)(X,Y,Z,W):=&S(X,W)T(Y,Z)+S(Y,Z)T(X,W)\\
&-S(X,Z)T(Y,W)-S(Y,W)T(X,Z)
\end{split}
\ee
for tangent vectors $X,Y,Z, W$.

Let $G$ be the Euclidean metric on $\R^{n+1}$.   Let $s=\log r$ and let
\be
\ol G= e^{-2s}G=r^{-2}G
\ee
where $r$ is the Euclidean distance to the origin.
 Let $M^n\subset \R^{n+1}$ be a hypersurface in $\R^{n+1}$ with induced metric $g, \ol g$ from $G, \ol G$ respectively. Denote:

\begin{itemize}
  \item $D, \ol D$, the covariant derivatives of $G, \ol G$;
  \item $\n, \on$  the covariant derivatives of $g, \ol g$.
\end{itemize}
.

\begin{lma}\label{l-curvature-conformal} Let $\oRm $ be the curvature tensor of $\ol G $. Then
\bee
\oRm=\frac12\ol G\owedge\ol G-(ds\otimes ds)\owedge \ol G.
\eee
\end{lma}
\begin{proof} Let $\mathrm{Rm}$ be the curvature tensor of $G$, then
\bee
\oRm=e^{-2s}\lf[\mathrm{Rm}+D^2 s\owedge G+(ds\otimes ds)\owedge G-\frac12|ds|^2_G G\owedge G\ri].
\eee
Now, $\mathrm{Rm}=0$, $D^2s=e^{-2s}G-2ds\otimes ds$, $|ds|^2_G=|Ds|^2_G=e^{-2s}$.
So
\bee
\begin{split}
\oRm=&e^{-2s}\lf[ \lf(e^{-2s}G-2ds\otimes ds\ri)\owedge G+(ds\otimes ds)\owedge G-\frac12 e^{-2s}G\owedge G\ri]
\\
=&\frac12 e^{-4s}G\owedge G-e^{-2s}(ds\otimes ds)\owedge G\\
=&\frac12\ol G\owedge\ol G-(ds\otimes ds)\owedge \ol G.
\end{split}
\eee
\end{proof}

 \begin{lma}\label{l-second fundamental}
 Let $M^n$ be a hypersurface in $\R^{n+1}$ with unit normal $\bn$ with respect to $G$, and let $\ol\bn=e^s\bn$ be the unit normal with respect to $\ol G$. Let  $\cK, \ol\cK$ be the second fundamental forms with respect to $G, \ol G$. Then
\bee
\cK 
=e^s\lf(\ol\cK +\ol\bn(s) \ol g \ri).
\eee
In particular, if $H, \ol H$ are the mean curvatures with respect to $G, \ol G$, then
$$
H=e^{-s}(\ol H+n\ol\bn(s)).
$$
If $M$ is minimal with respect to $G$, then $\ol H=-n\ol \bn(s)=:-n\omega$.
 \end{lma}

\begin{lma}\label{l-stability} Suppose $M^n$ is stably minimal in $(\R^{n+1},G)$,  then for any $\psi\in C_0^\infty(M)$, we have
\bee
\int_M |\on \psi|^2 d\mu_{\ol g}\ge \int_M W\psi^2 d\mu_{\ol g}
\eee
where
$$
W=:|\ol\cK|_{\ol g}^2-\frac{n^2+4n-4}{4 }\omega^2-\lf(\frac{n-2}2\ri)^2
$$
with $\omega^2\le1$.
\end{lma}
\begin{proof} This follows from \cite{ChodoshLiMinterStryker2024} and Lemma \ref{l-second fundamental}. In fact, by \cite[Proposition 3.10]{ChodoshLiMinterStryker2024},
\bee
\int_M |\on \psi|^2 d\mu_{\ol g}\ge \int_M \lf[(r^2|\mathcal{K}|_g^2-\frac{n(n-2)}2+\lf(\frac{n(n-2)}2-\frac{(n-2)^2}4\ri)|\n r|^2_g\ri]\psi^2 d\mu_{\ol g}
\eee
By Lemma \ref{l-second fundamental}
\bee
 |\cK|^2_g=e^{-4s}|\cK|^2_{\ol g}=e^{-2s}|\ol\cK+\ol \bn(s)\ol g|^2_{\ol g}=e^{-2s}(|\ol\cK|^2_{\ol g}-\frac1n\ol H^2)=e^{-2s}(|\ol\cK|^2_{\ol g}-n\omega^2).
 \eee
 Also
 \bee
 |\n r|^2_g=|\n e^s|^2_g=|\on s|^2_{\ol g}=1-(\ol \bn( s) )^2=1-\omega^2.
 \eee
Here we have used the fact that $|\ol Ds|^2=1$ and $|\ol Ds|^2=|\on s|^2_{\ol g}+(\ol \bn (s))^2$.
From these the result follows.
\end{proof}

\section{$\mu$-bubble}\label{s-mu}

By Lemma \ref{l-stability}, there is a positive function $w$ such that
\be\label{e-w}
-\ol\Delta  w= W w.
\ee
Here $\ol g$ is the metric on $M^n$ which is a stable minimal surface in $\R^{n+1}$, $\ol\Delta$ is it Laplacian.

Given $4>a\ge1$, $h$ is a suitable function. Let  $\Sigma^{n-1}\subset M^n$ be a $\mu$-bubble with respect to $w^a, h$ as defined in \eqref{e-mu-bubble}. In the following, $\on$, $\n^\Sigma$ are covariant derivatives of $M$ and $\Sigma$ and $\ol\Delta$ and $\Delta_\Sigma$ are the Laplacians of $M$, $\Sigma$, respectively. Denote the Ricci curvature of $M$ by $\ol\Ric^M$ and the Ricci curvature of $\Sigma$ by $\Ric^\Sigma$ etc. By the first and second variational formulas, see \cite{ChodoshLiMinterStryker2024,Mazet2024} we have:

\begin{lma}\label{l-bubble-1}\ \

\begin{enumerate}
  \item [(i)] First variational formula:
  \bee
H=h- a \vn (\log w).
\eee
$\vn$ is unit   normal and $H$ is the mean curvature.
  \item [(ii)] Second variational formula: For any smooth $\phi$ on $\Sigma$
\bee
\begin{split}
0\le &\int_{\Sigma}w^a\bigg\{(-\phi\Delta_\Sigma \phi-\bigg[(|A|^2+\ol\Ric^M(\vn,\vn)+a(\vn (\log w))^2 -a w^{-1}\on^2 w(\vn,\vn)+\vn(h)\bigg]\phi^2\\
&\ \ -a\la \on\log w,\n \phi\ra\phi\bigg\}
\end{split}
\eee
where $\on$ is the derivative of $M$, and $\ol \Ric$ is the Ricci curvature of $M$, $A$ is the second fundamental form of $\Sigma$.
\end{enumerate}

\end{lma}
By \cite{Mazet2024},   we have the following:

\begin{lma}\label{l-bubble-2} If $1\le a<4$, then for any smooth $\psi$ on $\Sigma$:
\bee
\begin{split}
\frac{4}{4-a}\int_\Sigma|\n^\Sigma \psi|^2\ge&\int_\Sigma \bigg[(|A|^2+\ol\Ric^M(\vn,\vn)+a(\vn (\log w))^2+ aW +aH \vn(\log w)+\vn(h)\bigg]\psi^2\\
=:&\int_\Sigma E\psi^2
\end{split}
\eee

\end{lma}
\begin{proof} This is just \cite[(12)]{Mazet2024}.
\end{proof}

Recall that in a Riemannian manifold,  the bi-Ricci curvature is defined as:
\be\label{e-biRicci}
\mathrm{BiRic}(e_1,e_2)=\Ric(e_1,e_1)+\Ric(e_2,e_2)-R(e_1,e_2,e_2,e_1)
\ee
for any orthonormal pair $e_1, e_2$.

\begin{lma}\label{l-biRicci} With same notation as in Lemma \ref{l-bubble-3}, let $e_1,\cdots, e_{n-1}$ be orthonormal in $\Sigma$. Then
\bee
\ol\Ric^M(\vn,\vn)=\ol{\mathrm{BiRic}}^M(\vn,e_1)-\Ric^\Sigma(e_1,e_1)+\sum_{j=2}^{n-1}(A_{11}A_{jj}-A_{1j}^2)
\eee
Here $\ol{\mathrm{BiRic}}^M$ is the bi-Ricci curvature of $M$, $\vn$ is a unit normal of $\Sigma$ in $M$ and $A$ is the second fundamental form with respect to $\vn$.
\end{lma}
\begin{proof}
\bee
\begin{split}
 \Ric^\Sigma(e_1,e_1)
=&\sum_{j=2}^{n-1}R^\Sigma_{ijj1}\\
=&\sum_{j=2}^{n-1}\oR^M_{1jj1}+\sum_{j=2}^{n-1}(A_{11}A_{jj}-A_{1j}^2)\\
=&\ol\Ric^M(e_1,e_1)-\oR^M(e_1,\vn,\vn,e_1)+\sum_{j=2}^{n-1}(A_{11}A_{jj}-A_{1j}^2)\\
=&\ol\Ric^M(e_1,e_1)+\ol\Ric^M(\vn,\vn)-\oR^M(e_1,\vn,\vn,e_1)-\ol\Ric^M(\vn,\vn)+\sum_{j=2}^{n-1}(A_{11}A_{jj}-A_{1j}^2)
\end{split}
\eee
So
\bee
\ol\Ric^M(\vn,\vn)=\ol{\mathrm{BiRic}}^M(e_1,\vn)+\sum_{j=2}^{n-1}(A_{11}A_{jj}-A_{1j}^2)-\Ric^\Sigma(e_1,e_1).
\eee
\end{proof}

\begin{lma}\label{l-bubble-3} For any $0\le \a\le 1$, and let $e_1,\cdots, e_{n-1}$ be an orthonormal frame on $\Sigma$. Let $E$ be as in Lemma \ref{l-bubble-2}, then
\bee
\begin{split}
E\ge&  aW+\a\ol{\mathrm{BiRic}}^M(e_1,\vn))+(1-\a) \ol\Ric^M(\vn,\vn)-\a \Ric^\Sigma(e_1,e_1)+Q\\
\end{split}
\eee
where $Q$ is a quadratic form defined in Theorem \ref{t-main}.
\end{lma}

\begin{proof} By Lemma \ref{l-bubble-1},  $\vn(\log w)=a^{-1}(h-H)$,  write
$$
\ol\Ric^M(\vn,\vn)=\a\ol\Ric^M(\vn,\vn)+(1-\a)\ol\Ric^M(\vn,\vn).
$$
By  Lemma \ref{l-bubble-3},
\bee
\begin{split}
E\ge& aW+\a\ol{\mathrm{BiRic}}^M(e_1,\vn))+(1-\a) \ol\Ric^M(\vn,\vn) -\a \Ric^\Sigma(e_1,e_1)\\
&+\a\sum_{j=2}^{n-1}(A_{11}A_{jj}-A_{1j}^2)+|A|^2+\frac1a(H-h)^2-H(H-h)-|\on h|\\
\ge&aW+\a\ol{\mathrm{BiRic}}^M(e_1,\vn))+(1-\a) \ol\Ric^M(\vn,\vn) -\a \Ric^\Sigma(e_1,e_1)\\
&+\a (A_{11}H-\sum_{j=1}^nA_{1j}^2)+|A|^2+\frac1a(H-h)^2-H(H-h)-|\on h|\\
\end{split}
\eee
Now:
\bee
\begin{split}
\a (A_{11}H-\sum_{j=1}^nA_{1j}^2)+|A|^2\ge& \a  A_{11}H +\sum_{j=2}^{n-1}A_{jj}^2+(1-\a)A_{11}^2\\
\ge&\a  A_{11}H+\frac1{n-2}(H-A_{11})^2+(1-\a)A_{11}^2\\
\ge&\frac1{n-2}H^2+(1-\a+\frac1{n-2})A_{11}^2+(\a-\frac2{n-2})A_{11}H.
\end{split}
\eee
So
\bee
\begin{split}
E\ge&aW+\a\ol{\mathrm{BiRic}}^M(e_1,\vn))+(1-\a) \ol\Ric^M(\vn,\vn) -\a \Ric^\Sigma(e_1,e_1)\\
 &+\frac1{n-2}H^2+(1-\a+\frac1{n-2})A_{11}^2+(\a-\frac2{n-2})A_{11}H\\
 &+\frac1a(H^2-2Hh+h^2)-H^2+Hh-|\on h|\\
 =&aW+\a\ol{\mathrm{BiRic}}^M(e_1,\vn))+(1-\a) \ol\Ric^M(\vn,\vn) -\a \Ric^\Sigma(e_1,e_1)\\
 &+\lf(\frac1{n-2}+\frac1a-1\ri)H^2+(1-\a+\frac1{n-2})A_{11}^2+(\a-\frac2{n-2})A_{11}H\\
 &+(1-\frac2a) Hh+\frac1ah^2  -|\on h|.
\end{split}
\eee
  From this the result follows with $z=A_{11}$.

\end{proof}

\section{
Proof of Theorem \ref{t-main}  }\label{s-main}
To prove Theorem \ref{t-main}, it remains to estimate $$ aW+\a\ol{\mathrm{BiRic}}^M(e_1,\vn)+(1-\a)\ol\Ric^M(\vn,\vn),$$ where $W$ is given by Lemma \ref{l-stability} and $\ol{\mathrm{BiRic}}^M$ is the bi-Ricci curvature and $\ol \Ric^M$ is the Ricci curvature of $M$.

\begin{lma}\label{l-W+BiRic} For $a\ge1$, $0\le \a\le1$,  let $\ol e_i, 1\le i\le n$ be an orthonormal basis for $M$ with respect to $\ol g$. Then
\bee
\begin{split}
&aW+\a \ol{\mathrm{BiRic}}^M(\ol e_1,\ol e_2)+(1-\a)\ol\Ric^M(\ol e_1,\ol e_1)\\
\ge &(n-2)\lf(\a-\frac{(n-2)a}4\ri)+P\\
\end{split}
\eee
where $P$ is a quadratic form as defined in Theorem \ref{t-main}.

\end{lma}

\begin{proof} Let $\ol e_i, 1\le i\le n$ be an orthonormal basis for $M$ with respect to $\ol g$.
Let $M$ be a hypersurface with second fundamental form $\ol\cK$ with respect to $\ol G$ unit normal $\ol \bn$ (then $\ol\bn=e^s\bn$ where $\bn$ is the unit normal with respect to $G$). Let $\omega=\ol\bn(s)=-\frac1n\ol H$. By Lemma \ref{l-curvature-conformal}, for $1\le i, j\le n$, with $i\neq j$, the curvature with respect to $\ol G$ is given by:
\bee
\oR(\ol e_i,\ol e_j,\ol e_j,\ol e_i)=1-\ol e_i^2(s)-\ol e_j^2(s).
\eee
By the Gauss equation:
\bee
 1-\ol e_i^2(s)-\ol e_j^2(s)=\oR^M_{ijji} -\ol\cK_{ii}\ol\cK_{jj}+\ol\cK_{ij}^2.
 \eee
 and so
 \be
 \oR^M_{ijji}= 1-\ol e_i^2(s)-\ol e_j^2(s)+\ol\cK_{ii}\ol\cK_{jj}-\ol\cK_{ij}^2.
 \ee
So the bi-Ricci of $M$ is:
\bee
\begin{split}
\ol{\mathrm{BiRic}}^M(\ol e_1,\ol e_2)=&\ol\Ric^M(\ol e_1,\ol e_1)+\ol\Ric^M(\ol e_2,\ol e_2)-\oR^M(\ol e_i,\ol e_j,\ol e_j, \ol e_i)\\
=&\sum_{j\neq 1, 1\le j\le n}\lf(1-\ol e_1^2(s)-\ol e_j^2(s)+\ol\cK_{11}\ol\cK_{jj}-\ol\cK_{1j}^2\ri)\\
&+\sum_{j\neq 2, 1\le j\le n}\lf(1-\ol e_2^2(s)-\ol e_j^2(s)+\ol\cK_{22}\ol\cK_{jj}-\ol\cK_{2j}^2\ri)
\\
&-( 1-\ol e_1^2(s)-\ol e_2^2(s)+\ol\cK_{11}\ol\cK_{22}-\ol\cK_{12}^2).
\end{split}
\eee
Now
\bee
\begin{split}
\sum_{j\neq 1, 1\le j\le n}\lf(1-\ol e_1^2(s)-\ol e_j^2(s)\ri)=&(n-1)-(n-2)\ol e_1^2(s)-\sum_{j=1}^n \ol e_j^2(s)\\
=&(n-1)-(n-2)\ol e_1^2(s)-|\on s|^2.
\end{split}
\eee
$$
\sum_{j\neq 1, 1\le j\le n}\lf( \ol\cK_{11}\ol\cK_{jj}-\ol\cK_{1j}^2\ri)=\ol\cK_{11}\ol H-\sum_{j=1}^n\ol\cK_{1j}^2.
$$
Similarly,
\bee
\begin{split}
\sum_{j\neq 2, 1\le j\le n}\lf(1-\ol e_2^2(s)-\ol e_j^2(s)\ri)=&(n-1)-(n-2)\ol e_2^2(s)-\sum_{j=1}^n \ol e_j^2(s)\\
=&(n-1)-(n-2)\ol e_2^2(s)-|\on s|^2.
\end{split}
\eee
$$
\sum_{j\neq 2, 1\le j\le n}\lf( \ol\cK_{22}\ol\cK_{jj}-\ol\cK_{2j}^2\ri)=\ol\cK_{22}\ol H-\sum_{j=1}^n\ol\cK_{2j}^2.
$$
Hence
\bee
\begin{split}
\ol{\mathrm{BiRic}}^M(\ol e_1,\ol e_2) =& (2n-3)-(n-3)(\ol e_1^2(s)+\ol e_2^2(s))-2|\on s|^2_{\ol g}+(\ol\cK_{11}+\ol \cK_{22})\ol H\\
&-\sum_{j=1}^n (\ol \cK_{1j}^2+\ol\cK_{2j}^2)-\ol\cK_{11}\ol\cK_{12}+\ol\cK_{12}^2\\
\end{split}
\eee
Similarly,

\bee
\begin{split}
 \ol\Ric^M(\ol e_1,\ol e_1)=& \sum_{j\ge 2}\ol{R}^M_{1jj1}\\
=& \sum_{j\ge 2}\lf(1-\ol e_1^2(s)-\ol e_j^2(s)+\ol\cK_{11}\ol\cK_{jj}-\ol\cK_{1j}^2\ri)\\
=&(n-1)-(n-2)\ol e_1^2-|\on s|^2+\ol \cK_{11}\ol H-\sum_{j=1}^n \ol\cK_{1j}^2.
\end{split}
\eee

So
\bee
\begin{split}
&\a \ol{\mathrm{BiRic}}^M(\ol e_1,\ol e_2)+(1-\a)\ol\Ric(\ol e_1,\ol e_1)\\
=&\a(2n-3)+(1-\a)(n-1)-\a(n-3)(\ol e_1^2(s)+\ol e_2^2(s))-(1-\a)(n-2)\ol e_1^2(s)\\
&-(2\a +(1-\a))|\on s|^2+(\ol\cK_{11}+\a \ol\cK_{22})\ol H-\sum_{j=1}^n (\ol\cK_{1j}^2+\a\ol\cK_{2j}^2)
-\a(\ol\cK_{11}\ol\cK_{22}-\ol\cK_{12}^2)\\
\ge &\a(2n-3)+(1-\a)(n-1)-(n-1)|\on s|^2 \\
&+(\ol\cK_{11}+\a \ol\cK_{22})\ol H-\sum_{j=1}^n (\ol\cK_{1j}^2+\a\ol\cK_{2j}^2)
-\a(\ol\cK_{11}\ol\cK_{22}-\ol\cK_{12}^2)\\
=&\a(n-2)+(n-1)\omega^2-n(\ol\cK_{11}+\a \ol\cK_{22})\omega-\sum_{j=1}^n (\ol\cK_{1j}^2+\a\ol\cK_{2j}^2)
-\a(\ol\cK_{11}\ol\cK_{22}-\ol\cK_{12}^2)
\end{split}
\eee
because    $0\le \a\le 1$ and $|\on s|^2=|\ol Ds|^2-(\ol\bn (s))^2=1-\omega^2$. By the definition of $W$ in Lemma \ref{l-stability}, since $a\ge 1$, $0\le \a\le 1$, we have
\bee
\begin{split}
&aW+\a \ol{\mathrm{BiRic}}^M(\ol e_1,\ol e_2)+(1-\a)\ol\Ric^M(\ol e_1,\ol e_1)\\
\ge &a\lf(|\ol\cK|^2-\frac{n^2+4n-4}4\omega^2-\lf(\frac{n-2}2\ri)^2\ri)\\
&+\a(n-2)+(n-1)\omega^2-n(\ol\cK_{11}+\a \ol\cK_{22})\omega-\sum_{j=1}^n (\ol\cK_{1j}^2+\a\ol\cK_{2j}^2)
-\a(\ol\cK_{11}\ol\cK_{22}-\ol\cK_{12}^2)\\
=&\a(n-2)-a\lf(\frac{n-2}2\ri)^2+\lf(n-1-\frac{a(n^2+4n-4)}4 \ri)\omega^2\\
&+a\sum_{i=3}^n\sum_{j=1}^n\ol\cK_{ij}^2+(a-1)\ol\cK_{11}^2+(a-\a)\ol\cK_{22}^2-\a\ol\cK_{11}\ol\cK_{22}-n(\ol\cK_{11}+\a \ol\cK_{22})\omega
\end{split}
\eee
Let $\ol\cK_{11}=x, \ol\cK_{22}=y$,
\bee
\begin{split}
\sum_{i=3}^n\sum_{j=1}^n\ol\cK_{ij}^2\ge& \frac1{n-2}(\ol H-x-y)^2=\frac1{n-2}(\ol H^2-2\ol H(x+y)+(x+y)^2)\\
=&\frac{n^2}{n-2}\omega^2+2\frac n{n-2}(x+y)\omega+\frac1{n-2}(x^2+2xy+y^2)
\end{split}
\eee
Hence, for $a\ge1, 0\le \a\le 1$, using the fact that $\ol H=-n \omega$ with $\omega=\ol{\bn}(s)$ 
\bee
\begin{split}
&aW+\a \ol{\mathrm{BiRic}}^M(\ol e_1,\ol e_2)+(1-\a)\ol\Ric(\ol e_1,\ol e_1)\\
\ge &\a(n-2)-a\lf(\frac{n-2}2\ri)^2+\lf(n-1-\frac{a(n^2+4n-4)}4 +\frac {a n^2}{n-2}\ri)\omega^2 \\
&+(a-1+\frac a{n-2})x^2+(a-\a+\frac a{n-2})y^2+(\frac {2a}{n-2}-\a)xy\\
&+n\lf((\frac {2a}{n-2}-1)x  +(\frac{2a}{n-2}-\a)y \ri)\omega\\
\end{split}
\eee
From this, the result follows.

\end{proof}

\begin{proof}[Proof of Theorem \ref{t-main}] Theorem \ref{t-main} follows from Lemmas \ref{l-bubble-2}, \ref{l-bubble-3}, \ref{l-W+BiRic}  where   we choose $e_1$ so that $\Ric^\Sigma(e_1,e_1)=\lambda_\Sigma$.

\end{proof}

\section{Applications}\label{s-application}

Let us discuss (i)--(iv) mentioned in section \ref{s-intro}.\vskip .2cm

Suppose $n=3$. Let $a=\a=1$.

\begin{enumerate}

\item[(i)] $(n-2)(\a-\frac{(n-2)a}4)=\frac 34$.

\item[(ii)]
\bee
P=\frac{27}4\omega^2
 +   x^2+ y^2+ xy
 +3\lf( x  + y \ri)\omega\\
\ge \frac{27}4\omega^2
 +  \frac34(x+y)^2
 +3\lf( x  + y \ri)\omega\ge0.
 \eee

 \item[(iii)]
 \bee
 Q=H^2+z^2-zH+h^2-hH
 \ge \frac34H^2+h^2-hH
 \ge \frac23 h^2.
 \eee

 \item[(iv)] Need to find suitable $h$ with $\e+\frac23 h^2-|\on h|\ge0$ for some $0<\e<\frac34.$

 \end{enumerate}

Moreover, it is obvious that:
\bee
\frac{4}{\a(4-a)}\le \frac{n-2}{n-3}.
\eee

Suppose $n=4$. Let  $a=\a=1$.

\begin{enumerate}
\item[(i)] $(n-2)(\a-\frac{(n-2)a}4)=1.$
\item[(ii)] $P=4\omega^2+\frac12x^2+\frac12y^2\ge0$.
\item[(iii)] $Q=\frac12H^2+\frac12z^2+h^2-hH \ge \frac12 h^2$.
\item[(iv)] Need to find suitable $h$ with $\frac12 h^2+\e-|\on h|\ge0$, for some $\e<1$.
\end{enumerate}

Moreover,
 \bee
\frac{4}{\a(4-a)}=\frac43\le \frac{4-2}{4-3}
\eee

Suppose $n=5$. As mentioned before, we cannot take $\a=1$. We may let $\a=\frac{4}{4-a}\cdot \frac 32$ and find a suitable $a$, if it exists, so that (ii), (iii) are true. As in Mazet's work \cite{Mazet2024}, one may take $a=11/10$. Then  let $\a=\frac{4}{4-a}\cdot \frac 32=\frac{80}{87}. $ A motivation for the choice of $a$ is as follows. If  we let  $\b=\frac1a$, then $aW+\ol{\mathrm{BiRic}}^M(\ol e_1,\ol e_2)=a(W+\b\ol{\mathrm{BiRic}}^M(\ol e_1,\ol e_2))$ and
 \bee
 \begin{split}
 &W+ \b\ol{\mathrm{BiRic}}^M(\ol e_1,\ol e_2)\\
 \ge& (n-2)(\b- \frac{ n-2}4) +\lf(\frac1{n-2}+\frac{4\b(n-1)-n^2-4n+4}{4n^2}\ri)\ol H^2\\
 &+\lf(\b-\frac{2}{n-2}\ri) \ol H x+(\frac1{n-2}-\frac\b4) x^2\
 \end{split}
 \eee
where $x=\ol\cK_{11}+\ol\cK_{22}$.
For $n=5$, we need $\b>\frac34$ so that the constant term is positive. Let $\b=\frac34+b$, $b>0$. Then
\bee
\begin{split}
W+  \b\ol{\mathrm{BiRic}}^M(\ol e_1,\ol e_2)\ge&3b +\lf(\frac{4b}{25}+\frac{13}{300}\ri)\ol H^2+(b+\frac1{12})\ol Hx+(\frac 7{48}-\frac b4)x^2
\end{split}
\eee
Want to find the largest $b$, if it exists, so that
\bee
(b+\frac1{12})^2-4\lf(\frac{4b}{25}+\frac{13}{300}\ri)(\frac 7{48}-\frac b4)\le0.
\eee
So we may choose
\bee
2b= \frac{5}{29\times 4} +\sqrt{(\frac{5}{29\times 4})^2+ \frac1{29}\frac{11}{6}}.
=0.2982
\eee
$
b=0.149,
$
$
\b=0.899
$
and   $a=1/\b=1.112347052280311.$ Hence we may try   $a=1.1=11/10.$  Let us consider (i)--(iv) for this choice of $a$, $\a$.

 \begin{enumerate}
   \item[(i)]  $(n-2)(\a-\frac{(n-2)a}4)=\frac{329}{29\times 40}$.
   \item[(ii)]

\bee
\begin{split}
   P
=&\frac{227}{120}\omega^2+\frac1{10\times 87}\lf(406x^2+476y^2-162 xy+5\omega(-232x-162 y)\ri)
\end{split}
\eee
Hence by Lemma \ref{l-Mazet} below:


  \bee
  \begin{split}
  406x^2+476y^2-162 xy+5\omega(-232x-162 y)
\ge&-1418.21875251078\omega^2,
\end{split}
\eee
and
\bee
\begin{split}
P\ge 1.891666666666667\omega^2-1.630136497138827\omega^2\ge0.
\end{split}
\eee

\item[(iii)]
\bee
 \begin{split}
 Q
 =&\frac1{11\times 87}\bigg[232H^2+396z^2+242zH-783 hH\bigg]+\frac{10}{11}h^2.
 \end{split}
 \eee
   Hence by Lemma \ref{l-Mazet},
   \bee
   \begin{split}
   232H^2+396z^2+242zH-783 hH\ge
   \ge &-785.8995869534254 h^2
   \end{split}
   \eee
   So
  \bee
 \begin{split}
 Q\ge   &\frac{10}{11}h^2-0.8212116896065052 h^2 \\
  \ge &0.0878792194844039 h^2.
   \end{split}
   \eee

   \item[(iv)] In \cite{Mazet2024}, a suitable $\wt h$ is constructed so that  $$
   a|\on \wt h|\le \frac3{20}+\frac1{22}\wt h^2.
   $$
   Let $\wt h=\lambda h$, then we have
   \bee
   |\n h|\le \frac1{a\lambda}\frac3{20}+\frac1{22 a}\lambda h^2.
   \eee
   Want
   $\frac1{22 a}\lambda=0.0878792194844039$, so $\lambda=0.0878792194844039\times 22a.$
So
\bee
\frac1{a\lambda}\frac3{20}=\frac1{0.0878792194844039\times 22a^2}\frac3{20}=0.0641205172260532.
\eee

\bee
\frac{329}{29\times 40}=0.2836206896551724>0.0641205172260532.
\eee
Hence we can find $\e<\frac{329}{29\times 40}$ and $h$ so that
\bee
|\on h|\le \e+0.0878792194844039 h^2.
\eee
 \end{enumerate}

 \vskip .2cm

In the above, we have used the following observation in \cite{Mazet2024}.
\begin{lma}\label{l-Mazet} Let $A$ be a positive definite $m\times m$ matrix and let $B$ be a vector in $\R^m$. Then for all $X\in \R^m$
\bee
\la AX,X\ra +\la B,X\ra\ge -\frac14 \la A^{-1}B, B\ra.
\eee
where $\la \ ,\ \ra$ is the standard inner product.
\end{lma}

\end{document}